\documentclass[12pt,reqno]{amsart}
\usepackage[utf8]{inputenc}
\usepackage[T1]{fontenc}
\usepackage[usenames, dvipsnames]{color}
\usepackage{ulem}

\usepackage{dsfont, amsfonts, amsmath, amssymb,amscd, stmaryrd, latexsym, amsthm, dsfont}
\usepackage[frenchb,english]{babel}
\usepackage{enumerate}
\usepackage{longtable}
\usepackage{geometry}
\usepackage{float}
\usepackage{tikz}
\usetikzlibrary{shapes,arrows}
\usepackage{float}
\usepackage{tikz}
\usepackage{xypic}
\usetikzlibrary{shapes,arrows}

\usepackage{pifont}
\usepackage{float}
\usepackage{tikz}
\usepackage{xypic}
\usetikzlibrary{shapes,arrows}

\newtheorem{theorem}{Theorem}[section]
\newtheorem{lemma}[theorem]{Lemma}
\newtheorem{proposition}[theorem]{Proposition}
\newtheorem{corollary}[theorem]{Corollary}

\theoremstyle{remark}
\newtheorem{remark}[theorem]{\bf Remark}

\usepackage[pagebackref]{hyperref}
\renewcommand*{\backref}[1]{}\renewcommand*{\backrefalt}[4]{\ifcase #1 (\tt not cited)\or (\tt cited on page~#2)\else (\tt cited on pages~#2)\fi}

\usepackage{hyperref}
\hypersetup{
	colorlinks=true,
	urlcolor=blue,
	citecolor=blue}
	\def\NN{\mathbb{N}}

\def\QQ{\mathbb{Q}}

\def\ZZ{\mathbb{Z}}

\def\G{\mathrm{G}}
\def\r2{\mathrm{rank}}

\begin{document}
\selectlanguage{english}
\title[On the maximal unramified pro-2-extension...]{On the maximal unramified pro-2-extension of $\ZZ_2$-extension of certain real biquadratic fields}
	
 	\author[M. M. Chems-Eddin]{Mohamed Mahmoud Chems-Eddin}
  \address{Mohamed Mahmoud CHEMS-EDDIN: Department of Mathematics, Faculty of Sciences Dhar El Mahraz,
  	Sidi Mohamed Ben Abdellah University, Fez,  Morocco}
  \email{2m.chemseddin@gmail.com}

	
\maketitle
	\begin{abstract}  For any   positive integer $n$, we show that there exists a real number field  $k$ (resp. $k'$) of degree $2^{n+2}$ whose $2$-class group is isomorphic to   $\ZZ/2\ZZ\times \ZZ/2\ZZ$ such that the Galois group of the maximal unramified extension of $k$ (resp. $k'$) over $k$ (resp. $k'$) is abelian (resp. non abelian, more precisely isomorphic to   $Q_8$ or $D_8$, the quaternion and the dihedral group of order $8$ respectively).
In fact, we construct the first examples in the literature of families of real biquadratic fields  for which the layers of the cyclotomic $\mathbb Z_2$-extension satisfy the previous   conditions and  whose unramified abelian $2$-Iwasawa modules are isomorphic to $\ZZ/2\ZZ\times \ZZ/2\ZZ$; hence these fields satisfy Greenberg’s conjecture.  
\end{abstract}
	
	   \section{\bf Introduction}

	 Let  $k$ be a   number field and let $\ell$ be a prime number. Denote by $\mathbf{C}l(k)$ (resp. $\mathbf{C}l_\ell(k)$, $E_k$) the class group (resp. the $\ell$-class group, the unit group) of $k$.
	 Let $k=k_0\subset k_1 \subset k_2 \subset\cdots \subset k_n \subset \cdots  \subset k_\infty$
	 be the cyclotomic $\ZZ_\ell$-extension of $k$. 
	 In particular, for an odd prime number $\ell$, let $ \QQ_{\ell, n}$   be the unique real subfield  of the cyclotomic field  $\QQ(\zeta_{\ell^{n+1}}) $ of degree $\ell^n$ over $\QQ$,  
	 and for $\ell=2$,   let $  \QQ_{2, n}$ be the field $\QQ(2\cos( {2\pi}/{2^{n+2}}))$, for all $n\geq 1$.
	 Then   $k_n= k\QQ_{\ell, n}$ and $k_\infty= \bigcup k_n$. The field $k_n$ is called the $n$th layer of the cyclotomic $\ZZ_\ell$-extension of $k$.
	  The inverse limit $X(k)=\varprojlim \mathbf{C}l_\ell(k_n)$
	 with respect to the norm maps is called the Iwasawa module for $k_\infty/k$. A spectacular result due to Iwasawa, affirms that  there exist integers $\lambda$,  $\mu\geq 0$ and  $\nu$, all independent of $n$, and  an integer $n_0$ such that:
	 \begin{eqnarray}\label{iwasawa}h_\ell(k_n)=\ell^{\lambda n+\mu \ell^n+\nu},\end{eqnarray}
	 for all $n\geq n_0$.  Here $h_\ell(k)$ denotes the $\ell$-class number of a number field $k$. The integers $\lambda$,  $\mu$ and  $\nu$ are called   Iwasawa invariants of $k_\infty/k$ (cf. \cite{iwasawa59}).

	 Greenberg conjectured that the invariants $\mu$ and $\lambda$ must be equal to $0$ for totally real number
	 fields (cf. \cite{Greenberg}) and it was further proved by Ferrero and Washington    (cf. \cite{FerreroWashington}) that the $\mu$-invariant always vanishes
	 for the cyclotomic $\mathbb Z_\ell$-extension when the number field is abelian over the field $\mathbb{Q}$ of rational numbers.
	 Various mathematicians have worked towards proving the vanishing of the $\lambda$-invariant for certain number
	 fields where the fundamental discriminant has small number of prime factors (cf. \cite{7,chemskatharina,1,jerrair2,2,3,8,4,5,6}).
	 Consider $\G_{k_\infty}=\mathrm{Gal}(\mathcal{L}(k_\infty)/k_\infty)$, the Galois group of the maximal unramified pro-$\ell$-extension $\mathcal{L}(k_\infty)$ of $k_\infty$. By class field theory, the Iwasawa module  $X(k)$ is isomorphic to the maximal abelian quotient group of $\G_{k_\infty}$, which is  
	 $\mathrm{Gal}(L(k_\infty)/k_\infty)$ the Galois group of the maximal unramified abelian pro-$\ell$-extension $L(k_\infty)$ of $k_\infty$. Note that   Greenberg's conjecture means that
	 $  \# X(k)=[L(k_\infty):k_\infty]$ is finite. The investigation of the structure of the group $\G_{k_\infty}$ for a given number field was 
	  of major importance in point of view of  many mathematicians (cf. \cite{d,e,a,c,ca,b,f}). 
	 
	 Note that all these investigations concern the case of (real and imaginary) quadratic  
	 fields and $\ell=2$.   In the present work,  we construct two  families of real biquadratic fields $F$ such that, for all $n\geq 0$,  
	 $\mathbf{C}l_2(F_n)\simeq \ZZ/2\ZZ\times \ZZ/2\ZZ$ (this implies that $\lambda=\mu=0$ for $F_\infty/F$), and such that the first family satisfy  $\G_{F_\infty}=\mathrm{Gal}(\mathcal{L}(F_\infty)/F_\infty)$ is abelian, while  the second family satisfy $\G_{F_\infty}$ is non-abelian, more precisely,   isomorphic to  $Q_8$ or $D_8$ (the quaternion group and dihedral   of order $8$ respectively). More precisely, we have:

	 \begin{theorem}[\bf The Main Theorem]\label{maintheorem}    
	 	Let  $q$,  $p$ and $s$ be  three distinct prime integers such that  $p\equiv   5\pmod 8$, $ q\equiv3\pmod 8$ and    $   s\equiv 3\pmod 4$ with  $\left(\frac{	p}{	q}\right)=\left(\frac{p}{s}\right)=1 $.    Put  $F:=\mathbb{Q}(\sqrt{pq}, \sqrt{ps} )$.    
	 	Then, for all $n\geq 0$, we have:
	 	$$	\mathbf{C}l_2(F_n) \simeq \ZZ/2\ZZ\times \ZZ/2\ZZ.$$
	 	Therefore, 	$X(F)   \simeq \ZZ/2\ZZ\times \ZZ/2\ZZ.$ 
	 	Moreover, letting $C(K):=\QQ(\sqrt {2p},  \sqrt{2s},\sqrt{q})$, we have:
	 	\begin{enumerate}[$1)$]
	 	   \item If $h_2(pqs)=h_2(C(K))=4$, then the group  $ \mathrm{Gal}(\mathcal{L}(F_n)/F_n)$ is abelian.
	 		\item If $h_2(pqs)=h_2(C(K))=8$,  then the group $ \mathrm{Gal}(\mathcal{L}(F_n)/F_n)$ is non abelian, more precisely, it is isomorphic to $Q_8$ or $D_8$, the quaternion and the dihedral group of order $8$ respectively.
	 	\end{enumerate}
	 	Here $h_2(d)$    denotes the $2$-class number of the real quadratic field $\mathbb{Q}(  \sqrt{d} )$  and  $h_2(k)$  denotes the $2$-class number  of a number  field  $k$.
	 \end{theorem}
	 
	 Furthermore, we prove the following proposition (cf. Proposition \ref{proptriq}) which gives an example of real triquadratic fields of the form $C(F):=\QQ(\sqrt{p }, \sqrt {q},  \sqrt{s})$ such that 
	 $\G_{C_\infty(F)}=\mathrm{Gal}(\mathcal{L}(C_\infty(F))/C_\infty(F))$, the Galois group of the maximal unramified pro-$2$-extension $\mathcal{L}(C_\infty(F))$ of $C_\infty(F)$
	 is cyclic non trivial.
	 \begin{proposition}\label{mainprop}	Let  $q$,  $p$ and $s$ be  three distinct odd prime integers such that  $p\equiv   5\pmod 8$, $ q\equiv3\pmod 8$ and    $   s\equiv 3\pmod 4$ with  $\left(\frac{	p}{	q}\right)=\left(\frac{p}{s}\right)=1 $. Then, for all $n\geq 0$, the $2$-class group of the field
	 	$$C_n(F):= \QQ(\sqrt{p }, \sqrt {q},  \sqrt{s},2\cos( {2\pi}/{2^{n+2}}))$$ is cyclic non trivial. If, moreover,   $h_2(pqs)=h_2(C(K))=4$, then for all $n\geq 0$, the $2$-class group of  $C_n(F)$ is isomorphic to $\ZZ/2\ZZ$.
	 \end{proposition}
	 
	 In the end of this paper, we give some numeral examples illustrating these results.
	 
	 \section{\bf Preliminaries}\label{sec2}
	 
	 Let us start by recalling some facts from class field theory that will be very useful for our proofs.  Let $k$ be  a number field and  $\mathbf{C}l_2(k)$ be its $2$-class group. Let  $k^{(1)}$  be the Hilbert $2$-class field of $k$, that is  the maximal unramified  abelian  extension
	 of $k$ whose degree over $k$ is a $2$-power. Put $k^{(0)} = k$ and let $k^{(i)}$ denote the Hilbert $2$-class field of $k^{(i-1)}$
	 for any integer $i\geq 1$. Then the sequence of fields
	 $$k=k^{(0)} \subset k^{(1)} \subset  k^{(2)}  \subset \cdots\subset k^{(i)} \cdots \subset  \bigcup_{i\geq 0} k^{(i)}=\mathcal{L}(k)$$
	 is called   the $2$-class field tower of $k$. If for all $i\geq1$,  $k^{(i)}\neq k^{(i-1)}$, the tower is said to be infinite, otherwise the tower is said to be  finite, and the minimal integer $i$ satisfying the condition $k^{(i)}= k^{(i-1)}$ is called the length of the tower. The field $\mathcal{L}(k)$ is called the maximal unramified pro-$2$-extension  of $k$ and for $k_\infty $ the cyclotomic $\ZZ_2$-extension of $k$, the group $\G_{k_\infty}=\mathrm{Gal}(\mathcal{L}(k_\infty)/k_\infty)$ is isomorphic to the inverse limit $\varprojlim \mathrm{Gal}(\mathcal{L}(k_n)/k_n)$ with respect to the restriction map.
	 
	 One of the most important and difficult problems in algebraic number theory is to decide whether or not a $2$-class field tower of a number field  is finite. Furthermore, the study of the structure of the Galois group of the tower is an open problem. Assume that $\mathbf{C}l_2(k)$ is isomorphic to  $\ZZ/2\ZZ \times \ZZ/2\ZZ$. In this case, the Hilbert $2$-class field tower of $k$ terminates in at most two steps (cf. \cite{Ki76}). In this  case, it is well-known that  $\mathrm{G}_k=\mathrm{Gal}(\mathcal{L}(k)/k)$ is isomorphic to one of the following $2$-groups
	 $V=\ZZ/2\ZZ \times \ZZ/2\ZZ$,  $Q_{2^m}$,   $D_{2^m}$,  and $S_{2^m}$ namely the Klein four group, the quaternion,   dihedral and semidihedral groups respectively,  of order $2^m$,  where $m\geq3$ and  $m\geq4$ for $S_{2^m}$. 
	 Let $x$ and $y$ be such that $\mathrm{G}_k= \langle x,y\rangle$.
	 The  commutator subgroup $\G_k'$ of $\G_k$  is always cyclic and $\G_k'= \langle x^2\rangle$. The group $\G_k$ possesses exactly three subgroups of index $2$ which are:
	 \begin{eqnarray*}
	 	H_1 = \langle x^2,  xy\rangle,\quad H_2 = \langle x^2,  y\rangle,\quad H_3 = \langle x\rangle. 
	 \end{eqnarray*}
	 Note also that for the two cases $Q_8$ and $V$, each $H_i$ is cyclic. For the case $D_m$, with $m>3$,  $H_2$ and $H_1$ are also dihedral. For $Q_{2^m}$, with $m>3$,  $H_2$ and $H_1$ are quaternion. Finally for $S_{2^m}$,
	 $H_2$ is dihedral whereas $H_1$ is quaternion. Furthermore,  if $\G_k$ is isomorphic to $V$ (resp.  $Q_{8}$),  then the subgroups $H_i$ are cyclic of order $2$ (resp. $4$).\label{grps props}
	 If $\G_k$ is isomorphic to $ Q_{2^m}$, with $m>3$,  $ D_{2^m}$, with $m>3$ or $ S_m$,  then   $H_3$ is cyclic and $H_i/H_i'$ is of type $(2, 2)$ for $i\in \{1,  2\}$,  where $H_i'$ is the commutator subgroup of $H_i$.

	 Let  $A(k)$, $B(k)$ and $C(k)$ be the subfields of $k^{(2)}$ fixed by $H_1$,  $H_2$ and $H_3$ receptively.
	 If $k^{(2)}\not=k^{(1)}$,  $\langle x^4\rangle$ is the unique subgroup of $\G_k'$ of index $2$.
	 Let  $L$ ($L$ is defined only if $k^{(2)}\not=k^{(1)}$)  be
	 the subfield of $k^{(2)}$ fixed by $\langle x^4\rangle$. Then $A(k)$, $B(k)$ and $C(k)$ are the three quadratic subextensions of $k^{(1)}/k$ and $L$ is the unique subfield of $k^{(2)}$ such that $L/k$ is a nonabelian Galois extension of degree $8$. For more details we refer the reader to \cite{acz, Ki76}. We  draw the following useful remarks.

	 \begin{remark}\label{rmk 1 preliminaries} 
	 	The $2$-class group of  $C(k)$ is cyclic.
	 \end{remark}
	 \begin{remark}\label{rem2}  	 
	 	The $2$-class groups of the three unramified quadratic  extensions of $k$ are cyclic if and only if  $k^{(1)}=k^{(2)}$ or  $k^{(1)}\not=k^{(2)}$ and $\G_k\simeq Q_{8}$. In the other cases the $2$-class group of only one unramified quadratic  extension is cyclic and the others are of type $(2,  2)$.
	 \end{remark}
	 
	 So the situation of the Hilbert $2$-class field tower  of $k$ and its three quadratic unramified extensions can be  schematized according to the two cases:
	 
	 \begin{enumerate}[$\bullet $]
	 	\item  If  $h_2(C(k))= 2$, we have:
	 	\begin{figure}[H]	$$
	 		\begin{tikzpicture} [scale=1.2]
	 			\node (k)  at (0,  0) {$k$};
	 			\node (K)  at (-1,  1) {$A(k)$};
	 			\node (F)  at (1,  1) {$B(k)$};
	 			\node (L)  at (0,  1) {$C(k)$};
	 			\node (L*)  at (0,  2) {$A(k)^{(1)}=B(k)^{(1)}=  C(k)^{(1)}=k^{(1)}=k^{(2)}$};
	 			\draw (k) --(K)  node[scale=0.4,  midway,  below right]{};
	 			\draw (k) --(F)  node[scale=0.4,  midway,  below right]{};
	 			\draw (k) --(L)  node[scale=0.4,  midway,  below right]{\Large \bf2};
	 			\draw (k) --(L)  node[scale=0.4,  midway,  below right]{};
	 			\draw (L) --(L*)  node[scale=0.4,  midway,  below right]{\Large \bf2};
	 			\draw (K) --(L*)  node[scale=0.4,  midway,  below right]{};
	 			\draw (F) --(L*)  node[scale=0.4,  midway,  below right]{};
	 		\end{tikzpicture}    		 $$
	 		\caption{The case $h_2(C(k))= 2$}\label{fig1}\end{figure}
	 	\item If $h_2(C(k))\geq 4$, we have:
	 	\begin{figure}[H] $$	\begin{tikzpicture} [scale=1.2]
	 			\node (k)  at (0,  0) {$k$};
	 			\node (K)  at (-1,  1) {$A(k)$};
	 			\node (F)  at (1,  1) {$B(k)$};
	 			\node (L)  at (0,  1) {$C(k)$};
	 			\node (L*)  at (0,  2) {$k^{(1)}$};
	 			\node (Lpq1)  at (0,  3) {$A(k)^{(1)}=B(k)^{(1)}$};
	 			\node (Lpq1*1)  at (0,  4.75) {$A(k)^{(2)}=B(k)^{(2)}=C(k)^{(1)}=k^{(2)}$};
	 			\draw (k) --(K)  node[scale=0.4,  midway,  below right]{};
	 			\draw (k) --(F)  node[scale=0.4,  midway,  below right]{};
	 			\draw (k) --(L)  node[scale=0.4,  midway,  below right]{\Large \bf2};
	 			\draw (k) --(L)  node[scale=0.4,  midway,  below right]{};
	 			\draw (L) --(L*)  node[scale=0.4,  midway,  below right]{\Large \bf2};
	 			\draw (K) --(L*)  node[scale=0.4,  midway,  below right]{};
	 			\draw (F) --(L*)  node[scale=0.4,  midway,  below right]{};
	 			\draw (L*) --(Lpq1)  node[scale=0.4,  midway,  below right]{\Large \bf \hspace{-0.2cm}\;\;\;2};
	 			\draw (Lpq1) --(Lpq1*1)  node[scale=0.4,  midway,  below right]{\Large \bf$\frac{ h_2( C(k) )}{4}$};
	 		\end{tikzpicture} $$	
	 		\caption{The case  $h_2(C(k))\geq 4$}\label{fig2}\end{figure}
	 \end{enumerate}


	 \noindent\textbf{Notations:}  Keep the above assumptions and notations. The $n$th  layer  of the cyclotomic $\mathbb{Z}_2$-extension of the fields $A(k)$, $B(k)$ and $C(k)$ will be denoted $A_n(k)$, $B_n(k)$ and $C_n(k)$ respectively.
	 
	 \bigskip
	 Let us now recall some other   useful lemmas. The following result is called the ambiguous class number formula.
	 
	 \begin{lemma}[\cite{Qinred}, Lemma 2.4]\label{AmbiguousClassNumberFormula} Let $k/k'$ be a quadratic extension of number fields. If the class number of $k'$ is odd, then the  rank of the  $2$-class group of $k$ is given by
	 	$$\r2({\mathbf{C}l_2(k)})=t-1-e,$$
	 	with  $t$ is the number of  ramified primes (finite or infinite) in the extension  $k/k'$ and $e$ is  defined by   $2^{e}=[E_{k'}:E_{k'} \cap N_{k/k'}(k^*)]$.
	 \end{lemma}

	 The following lemma is a particular case of Fukuda's Theorem  \cite{fukuda}.

	 \begin{lemma}[\cite{fukuda}]\label{lm fukuda}
	 	Let $k_\infty/k$ be a $\mathbb{Z}_2$-extension and $n_0$  an integer such that any prime of $k_\infty$ which is ramified in $k_\infty/k$ is totally ramified in $k_\infty/k_{n_0}$.
	 	\begin{enumerate}[\rm $1)$]
	 		\item If there exists an integer $n\geq n_0$ such that   $h_2(k_n)=h_2(k_{n+1})$, then $h_2(k_n)=h_2(k_{m})$ for all $m\geq n$.
	 		\item If there exists an integer $n\geq n_0$ such that $\r2( \mathbf{C}l_2(k_n))= \r2(\mathbf{C}l_2(k_{n+1}))$, then
	 		$\r2(\mathbf{C}l_2(k_{m}))= \r2(\mathbf{C}l_2(k_{n}))$ for all $m\geq n$.
	 	\end{enumerate}
	 \end{lemma}
	 
	 
	 \medskip
	 
	 The   following class number formula for   multiquadratic number fields  is usually attributed to Kuroda \cite{Ku-50} or Wada \cite{Wa-66}, but it goes back to Herglotz \cite{He-22}.
		 \medskip
	 \begin{lemma}[\cite{Ku-50}]\label{wada's f.}
	 	Let $k$ be a multiquadratic number field of degree $2^n$, with  $n\geq 2$ is an integer,  and $k_i$ the $s=2^n-1$ quadratic subfields of $k$. Then
	 	$$h(k)=\frac{1}{2^v}q(k)\prod_{i=1}^{s}h(k_i),$$
	 	with  $ q(k):=[E_k: \prod_{i=1}^{s}E_{k_i}]$ and   $$     v=\left\{ \begin{array}{cl}
	 		n(2^{n-1}-1); &\text{ if } k \text{ is real, }\\
	 		(n-1)(2^{n-2}-1)+2^{n-1}-1 & \text{ if } k \text{ is imaginary.}
	 	\end{array}\right.$$
	 \end{lemma}
	 
	 To use this lemma, we shall need the following   values  of $2$-class numbers of certain quadratic fields. 
	 \begin{remark}\label{clval}Let $p\equiv 5\pmod 8$ and $q\equiv s\equiv 3\pmod 4$ be three distinct prime numbers. We have:
	 	\begin{enumerate}[$\bullet$]
	 		\item $h_2(2)=h_2(p)=h_2(q)=h_2(2q)=h_2(qs)=1$ (cf. \cite[Corollary 18.4]{connor88}).
	 		
	 		\item  $h_2(pq)=h_2(ps)=h_2(2pq)=2$. If moreover  $s\equiv 3\pmod 8 $ or $q\equiv 3\pmod 8$, then $h_2(2sq)=2$ (cf. \cite[Corollary 19.7]{connor88}).
	 	\end{enumerate} 
	 \end{remark}   
	 
	 Now let us  recall the following  method given in    \cite{Wa-66}, that describes a fundamental system  of units of a  multiquadratic field $k_0$. Let  $\sigma_1$ and 
	 $\sigma_2$ be two distinct elements of order $2$ of the Galois group of $k_0/\mathbb{Q}$. Let $k_1$, $k_2$ and $k_3$ be the three subextensions of $k_0$ invariant by  $\sigma_1$,
	 $\sigma_2$ and $\sigma_3= \sigma_1\sigma_2$, respectively. Let $\varepsilon$ denote a unit of $k_0$. Then \label{algo wada}
	 $$\varepsilon^2=\varepsilon\varepsilon^{\sigma_1}  \varepsilon\varepsilon^{\sigma_2}(\varepsilon^{\sigma_1}\varepsilon^{\sigma_2})^{-1},$$
	 and we have, $\varepsilon\varepsilon^{\sigma_1}\in E_{k_1}$, $\varepsilon\varepsilon^{\sigma_2}\in E_{k_2}$  and $\varepsilon^{\sigma_1}\varepsilon^{\sigma_2}\in E_{k_3}$.
	 It follows that the unit group of $k_0$  
	 is generated by the elements of  $E_{k_1}$, $E_{k_2}$ and $E_{k_3}$, and the square roots of elements of   $E_{k_1}E_{k_2}E_{k_3}$ which are perfect squares in $k_0$.

	 \section{\bf The proof of the main theorem}
	 
	 The proof of our main theorem  relies on the following  lemmas and propositions.  	Let  $q$,  $p$ and $s$ be  three distinct prime integers such that  $p\equiv   5\pmod 8$, $ q\equiv3\pmod 8$ and    $   s\equiv 3\pmod 4$ with  $\left(\frac{	p}{	q}\right)=\left(\frac{p}{s}\right)=1 $.
	 Put   $F=\mathbb{Q}(\sqrt{pq}, \sqrt{ps} )$ and  $K=\QQ(\sqrt {2pq},  \sqrt{ps})$. Let   $\varepsilon_{d}$ (resp. $h_2(d)$) be the fundamental unit (resp. $2$-class number) of a real quadratic field $\mathbb{Q}(\sqrt{d})$.  We have:

	 \begin{lemma}\label{cn}  
	 	The $2$-class groups of $K$ and $F$ are isomorphic   to $\ZZ/2\ZZ\times \ZZ/2\ZZ$.
	 \end{lemma}  
	 \begin{proof}$\:$\\
	 	\noindent\ding{229} Let us start by proving that $\mathbf{C}l_2(K)\simeq \ZZ/2\ZZ\times \ZZ/2\ZZ$.

	 	$\bullet$ Notice that if $s\equiv 7\pmod 8$ (resp. $s\equiv 3\pmod 8$), then according to \cite[Lemmas 5 and 7]{chemszekhniniazizilambdas} (resp. \cite[Lemmas 4]{acztrends}), we have 
	 	$\sqrt{\varepsilon_{2sq}}=\frac12(2y_1\sqrt{s} +y_2\sqrt{2q})$ (resp. $\sqrt{\varepsilon_{2sq}}=\frac12(y_1'\sqrt{2}  +2y_2'\sqrt{sq})$) for some integers $y_1$ and $y_2$ (resp. $y_1'$ and $y_2'$).


	 	$\bullet$ According to \cite[Lemmas 2.3]{ChemsUnits9}, $\sqrt{\varepsilon_{2pq}}= \frac12(a_1\sqrt{2p} +2a_2\sqrt{q})$, for some integers $a_1$ and $a_2$.
	 	

	 	$\bullet$ Since  $\left(\frac{2}{p}\right)=-1$, then it is easy to deduce from  \cite[The proof of Proposition 3.3]{aztaous} that  $\sqrt{\varepsilon_{ps}}=   (c_1\sqrt{p} +c_2\sqrt{s})$, for some integers $c_1$ and $c_2$.

	 	It follows that 
	 	$\sqrt{\varepsilon_{2pq}}$, $\sqrt{\varepsilon_{2sq}}$, $\sqrt{ \varepsilon_{2sq}\varepsilon_{2pq}}$, $\sqrt{\varepsilon_{ps}}$, $\sqrt{ \varepsilon_{2pq}\varepsilon_{ps}} \not\in K=\QQ(\sqrt {2pq},  \sqrt{ps})$
	 	 and there is exactly one element $\eta \in \{    \sqrt{ \varepsilon_{2sq}\varepsilon_{ps}},\sqrt{\varepsilon_{2pq} \varepsilon_{2sq}\varepsilon_{ps}} \}$ such that $\eta\in  K$, more precisely, $ \eta  =   \sqrt{ \varepsilon_{2sq}\varepsilon_{ps}}$ if and only if $s\equiv 7\pmod8$. Therefore,
	 	 a fundamental system of units of $F$ is 
	 	 $\{\varepsilon_{2pq}, \varepsilon_{ps} ,  \eta\}$ and so
	 	   $q(K)=  2$. On the other hand,  using Lemma \ref{wada's f.} and Remark \ref{clval}, we get:
	 	\begin{eqnarray*}
	 		h_2(K)&=&\frac{1}{4}q(K)h_2(2pq)h_2(2qs)h_2(ps),\\
	 		&=&\frac{1}{4}   \cdot 2 \cdot 2 \cdot 2 \cdot 2= 4.  
	 	\end{eqnarray*}
	 	Notice that $F_1= \mathbb{Q}(\sqrt{pq}, \sqrt{ps} , \sqrt{2})$  and $\QQ(\sqrt {2q},  \sqrt{s}, \sqrt{p })$ are two different   unramified extensions of $K$. So, by class field theory,  $\r2(\mathbf{C}l_2(K))\geq 2$.  Therefore, 
	 	$\mathbf{C}l_2(K)\simeq \ZZ/2\ZZ\times \ZZ/2\ZZ$.
	 	
	 	\noindent\ding{229} Now let us  prove that $\mathbf{C}l_2(F)\simeq \ZZ/2\ZZ\times \ZZ/2\ZZ$. We have:


	 	$\bullet$ According to \cite[Lemmas 5 and 7]{chemszekhniniazizilambdas} and \cite[Lemmas 4]{acztrends}, we have $\sqrt{ \varepsilon_{sq}}=y_1\sqrt{s} +y_2\sqrt{q}$,  for some integers $y_1$ and $y_2$.


	 	$\bullet$ According to \cite[Lemmas 2.3]{ChemsUnits9}, $\sqrt{\varepsilon_{pq}}= b_1\sqrt{p} +b_2\sqrt{q}$, for some integers $b_1$ and $b_2$.
	 	
	 	Thus, $\sqrt{\varepsilon_{pq} }$, $\sqrt{ \varepsilon_{sq}}$,  $\sqrt{ \varepsilon_{ps}}\not\in F=\mathbb{Q}(\sqrt{pq}, \sqrt{ps} )$  and  $\sqrt{\varepsilon_{pq}\varepsilon_{sq}}$,  $\sqrt{\varepsilon_{pq}\varepsilon_{ps}}\in F$. Therefore, a fundamental system of units of $F$ is 
	 	$\{\varepsilon_{pq}, \sqrt{\varepsilon_{pq}\varepsilon_{sq}} ,   \sqrt{\varepsilon_{pq}\varepsilon_{ps}}\}$. Thus, $q(F)=4$.
	 	It follows, by class number formula (cf. Lemma \ref{wada's f.} and Remark \ref{clval}), that we have:    \begin{eqnarray}\label{hf=4}
	 		 h_2(F)=\frac{1}{4}q(F)h_2(pq)h_2(ps)h_2(sq)=4.
	 	\end{eqnarray}  On the other hand, notice that $ \mathbb{Q}(\sqrt{sq})$ is a subfield of $F$ whose  class number is odd (cf. Remark \ref{clval}). So Lemma \ref{AmbiguousClassNumberFormula} gives $\r2(\mathbf{C}l_2(F))= t-1-e$, where $t$ $(=3$ or $4$ according to whether $s\equiv 7\pmod 8$ or not$)$ is the number of ramified primes in  $F/ \mathbb{Q}(\sqrt{sq})$ and the index $e$ is such that
	 	$(E_{\mathbb{Q}(\sqrt{sq})}:E_{\mathbb{Q}(\sqrt{sq})}\cap N_{F/\mathbb{Q}(\sqrt{sq})}(F))=2^{ e}$.  
	  As  by \cite[Théorème 3.3 et Théorème 3.4]{Azmouh2-rank} $e=0$ or $1$ according to whether   $s\equiv 7\pmod 8$ or not, we have $\r2(\mathbf{C}l_2(F))=2$. This completes the proof.
	 \end{proof}

	 \medskip
	In the  proof of the following lemma, we use the    properties of $\left(\frac{\cdot,\,\cdot}{ \mathfrak p}\right)$, the norm residue symbol. For more details concerning these properties,  
	 we refer the reader to 
	 \cite[Chapter II, Theorem 3.1.3]{grasbook} or 
	 \cite[Chapter X]{Fuller}.
	 
	 	 \medskip

	 \begin{lemma}\label{lemmarank}
	 	The   $2$-class group of $F_1=\QQ(\sqrt{pq}, \sqrt{ps}, \sqrt{2})$ is  isomorphic   to $\ZZ/2\ZZ\times \ZZ/2\ZZ$.
	 \end{lemma}
	 \begin{proof} 
	 	Put $L=\QQ(\sqrt{sq}, \sqrt{2})$. Notice that $F_1=L(\sqrt{ps})$. Pursuant to \cite[p. 19]{acztrends} and \cite[Corollaries 1 and 2]{chemszekhniniazizilambdas}, we have  $E_{L}=\left\langle -1, \varepsilon_{2},  \varepsilon_{ qs},
	 	\sqrt{\varepsilon_{2qs}}\text{ or }\sqrt{\varepsilon_{qs}\varepsilon_{2qs}}\right\rangle $, according to whether $s\equiv 3\pmod 8$ or not. Note that  $h_2(L)=1$ (cf. \cite[Corollary 21.4]{connor88}). Let us assume that  $s\equiv 3\pmod 8$.
	 	 	Therefore, by Lemma \ref{AmbiguousClassNumberFormula}, the rank of the $2$-class group of $F_1$ is  $\r2(\mathbf{C}l_2(F_1))=t-1-e$, where ${e}$ is defined by $(E_{L}:E_{L}\cap N_{F_1/L}(F_1))=2^{ e}$ and $t=4$ is the number of ramified primes in $F_1/L$. Thus, we have  $\r2(\mathbf{C}l_2(F_1))=3-e$.
	  	Let $\mathfrak p_{k}$ be a prime ideal of $k$ above $p$, where $k$ is a subfield of $L$. Notice that $p$ decomposes in $\mathbb{Q}(\sqrt{qs})$ and there are exactly $2$ prime ideals
	 	of $L$ laying above $p$. Using the properties of norm residue symbols, we obtain:
	 		\begin{eqnarray*}
	 		\left(\frac{\varepsilon_2,\,ps}{ \mathfrak p_{L}}\right)=\left(\frac{\varepsilon_2,\,p}{ \mathfrak p_{L}}\right)&=&\left(\frac{N_{{L}/\mathbb Q(\sqrt{qs})}(\varepsilon_2),\,p}{ \mathfrak p_{\mathbb Q(\sqrt{qs})}}\right) 
	 		 = \left(\frac{-1,\,p}{ \mathfrak p_{\mathbb Q(\sqrt{qs})}}\right) =\left(\frac{-1,\,p}{ p}\right)  =1, 
	 	\end{eqnarray*}
	 	 	and in a similar way $\left(\frac{\varepsilon_{qs},\,ps}{ \mathfrak p_{L}}\right)=\left(\frac{-1,\,ps}{ \mathfrak p_{L}}\right)=1$.
	 	It follows that $e\leq 1$ and so  $\r2(\mathbf{C}l_2(F_1))\geq 2$. According to Lemma  \ref{cn}, the $2$-class group of $K$ is isomorphic  to $\ZZ/2\ZZ\times \ZZ/2\ZZ$ and since $F_1/K$ is an  unramified quadratic extension, it follows that the $2$-class group of $F_1$ is isomorphic to $\ZZ/2\ZZ\times \ZZ/2\ZZ$ (cf. Remark \ref{rem2}).
	 	We   proceed similarly for the case $s\equiv 7\pmod 8$. This completes the proof.
	 \end{proof}

	 \begin{corollary}
	 	  Put $ A(K):=\QQ(\sqrt {2q},  \sqrt{s},\sqrt{p})$ and $ C(K):=\QQ(\sqrt {2p},  \sqrt{2s},\sqrt{q})$. 
	 	Then, the  $2$-class group  of $A(K)$ $($resp.  $C(K)$$)$ is isomorphic to $\ZZ/2\ZZ\times \ZZ/2\ZZ$ $($resp. is cyclic$)$.
	 \end{corollary}
	 \begin{proof}
	 	Put $k=\QQ(\sqrt {2q},  \sqrt{ps})$. Notice that $ps\equiv 3\pmod 4$.  As $2$ is totally ramified in $k$ and according to \cite[Theorem 4.4]{OuyangZhang} $k$ admits an unramified    extension of degree $8$ of the form $\QQ(\sqrt {2 },\sqrt {q},  \sqrt{p},\sqrt{s}, \alpha)$, where $\alpha$ is defined in \cite[Theorem 4.4]{OuyangZhang}, the field  $ A(K)$ admits a biquadratic unramified extension. So, by class field theory, we have  $\r2(\mathbf{C}l_2(A(K)))\geq 2$. Therefore, since $ A(K)/K$ is unramified quadratic extension such that the $2$-class group of $K$ is    isomorphic to $\ZZ/2\ZZ\times \ZZ/2\ZZ$ (cf. Lemma \ref{cn}), we deduce that     $ \mathbf{C}l_2(A(K))\simeq \ZZ/2\ZZ\times \ZZ/2\ZZ$ (cf. Remark \ref{rem2}). Since $F_1/K$ is unramified, we have  $ \mathbf{C}l_2(C(K))$ is cyclic (cf. Lemma \ref{lemmarank} and   Remark \ref{rem2}).
	 \end{proof}

	 \begin{lemma}\label{cn12hpqs} 
	 	The $2$-class group of  $ \QQ(\sqrt{p }, \sqrt {q},  \sqrt{s}  )$ is isomorphic to  $  \mathbb{Z}/m\mathbb{Z},$ 
	 	where $m=\frac{1}{2}h_2(pqs)$.
	 \end{lemma}
	 \begin{proof} Put $L'=\QQ(\sqrt{p q},   \sqrt{s}  )$.  Without loss of generality, we may assume that $\left(\frac{s}{q}\right)=-1$ (if not we exchange the places of $q$ and $s$).
	 	Notice that   the class number of $\QQ(\sqrt{s})$  is odd (cf. Remark \ref{clval}). So by the ambiguous class number formula
	 	$\r2(\mathbf{C}l_2(L'))=3-1-e=2-e$, where ${e}$ is such that $(E_{\QQ(\sqrt{s})}:E_{\QQ(\sqrt{s})}\cap N_{L'/\QQ(\sqrt{s})}(L'))=2^{ e}$.
	 	According to 	 \cite[Proposition 2.6]{OuyangZhang}, we have $\varepsilon_s=2u^2$ for some $u\in \QQ(\sqrt{s}  )$. Thus,   $\left(\frac{\varepsilon_s,\,pq}{ \mathfrak p_{\QQ(\sqrt{s})}}\right)=\left(\frac{2u^2,\,p }{ \mathfrak p_{\QQ(\sqrt{s})}}\right)=\left(\frac{2,\,p }{ \mathfrak p_{\QQ(\sqrt{s})}}\right)= \left(\frac{2}{ p}\right)  =-1$.
	 	Therefore, $e\geq 1$, but from the fact that the class number of $ \QQ(\sqrt{p q},   \sqrt{s})$ is even   (cf. \cite[Corollary 21.4]{connor88}), we deduce that
	 	$\r2(\mathbf{C}l_2(L'))=1$. Thus, $\mathbf{C}l_2(L')$ is a cyclic group. As $S/L'$, where $S=\QQ(\sqrt{p }, \sqrt {q},  \sqrt{s}  )$, is a quadratic unramified extension, it follows, by class field theory,  that $\mathbf{C}l_2(S)$ is cyclic and $h_2(S)=\frac{1}{2} h_2(L')$.
	 	Therefore, by Lemma \ref{wada's f.} and Remark \ref{clval}, we have :
 		\begin{eqnarray*}
 		h_2(L')&=&\frac{1}{4}q(L')h_2(pq)h_2(s)h_2(pqs)
 		=\frac{1}{4} q(L') \cdot 2 \cdot 1 \cdot h_2(pqs) =\frac{1}{2} q(L')    \cdot h_2(pqs)   .  
 	\end{eqnarray*} 
	 	We note that according to \cite[Lemma 2.3 (2)(i)]{ChemsUnits9},	 $\varepsilon_{pq}=a+b\sqrt{pq}$ for some integers $a$ and $b$ such that $(a\pm1)$ is not a square in $\NN$. So by
	 	\cite[Proposition 3.2]{AziziZekhnini17}, we have $\{\varepsilon_{s},\varepsilon_{pq}, \sqrt{\nu }\}$ is a fundamental system of units of $L'$, where $\nu= {\varepsilon_{pqs}}$,
	 	$ {\varepsilon_{s}\varepsilon_{pq}}$ or $ {\varepsilon_{s}\varepsilon_{pq}\varepsilon_{pqs}}$. It follows that 	$q(L')=2$. Therefore,  $h_2(S)=	\frac{1}{2}\cdot (\frac{1}{2} q(L')    \cdot h_2(pqs))=\frac{1}{2} h_2(pqs).$ So the result.
	 \end{proof}
	 
	 The following corollary follows from the above proof.
	 \begin{corollary}
	 	The $2$-class groups of $\QQ(\sqrt{p q},   \sqrt{s}  )$ and $\QQ(\sqrt{p s},   \sqrt{q}  )$ are cyclic of order $h_2(pqs)$.
	 \end{corollary}
	 
	 \medskip
	 
	 The following proposition gives  a family of real triquadratic number fields of the form $\QQ(\sqrt{p }, \sqrt {q},  \sqrt{s})$, where $p$, $q$ and $s$ are odd primes  (called Fröhlich multiquadratic fields)
	 whose Iwasawa module is cyclic non trivial.   To our knowledge, the only previously known results about the Iwasawa module of such fields are in
	 \cite[p. 1206, Corollary]{Ajn},  which provide  a family of these fields whose Iwasawa module is trivial.
	 

	\begin{proposition}\label{proptriq}
	  Let      $C_n(F)= \QQ(\sqrt{p }, \sqrt {q},  \sqrt{s},2\cos( {2\pi}/{2^{n+2}}))$, where $n$ is a positive integer. Then $2$-class group of $C_n(F)$ is cyclic non trivial.
	\end{proposition}
	 \begin{proof}
	 	Recall that by Lemma \ref{cn}, the $2$-class group of $K$ is   isomorphic to $\ZZ/2\ZZ\times \ZZ/2\ZZ$. Thus,   $C_1(F)=\QQ(\sqrt{p }, \sqrt {q},  \sqrt{s},\sqrt{2})$     is the Hilbert $2$-class field of $K$. In fact, $C_1(F)/K$
	 	is an unramified abelian extension of degree $4$. So according to the theoretic properties in Section \ref{sec2}, 
	 	 the  $2$-class group  of $C_1(F)$ is cyclic. Note that the extension $C_1(F)/C(F)$, with  $C(F)=\QQ(\sqrt{p }, \sqrt {q},  \sqrt{s}  )$,  is totally ramified.   
	 	 As $h_2(C(F))=\frac12 h_2(pqs)$ is divisible by $2$ (cf. \cite[Corollary 19.7]{connor88}), it follows   $h_2(C_1(F))$ is divisible by $2$ and so the $2$-class groups of 
	 	 $C(F)$ and $C_1(F)$ are cyclic non trivial.
	 	  Hence the result by Fukuda's Theorem (cf. Lemma \ref{lm fukuda}).
	 \end{proof}

	 \bigskip
	 	With the notations fixed at the beginning of the previous section, we have the following diagram. Note that a solid line indicates that the extension is unramified, whereas a dashed line indicates that the extension is ramified at   $2$.


	 \newpage
	 
	 \rotatebox{270}{
	 	{

	 		\begin{minipage}{20cm}
	 			\begin{center}
	 				
	 				\begin{figure}[H]
	 					\caption{Layers of  $\mathbb Z_2$-extensions:}\label{fig3}
	 					{ \footnotesize
	 						\hspace*{-3cm}
	 						\begin{tikzpicture} [scale=1.2]
	 							\node (F)  at (-0.5,  0) {$\qquad F=\QQ(\sqrt{pq},\sqrt{ps})$};
	 							\node (K)  at (-3,  0) {$ K=\QQ(\sqrt {2pq},  \sqrt{ps})$};
	 							\node (CK)  at (2.25,  2) {$ C(K)=\QQ(\sqrt {2p},  \sqrt{2s},\sqrt{q})$};
	 							
	 							\node (F1)  at (-0.5,  2) {$F_1$};
	 							\node (AF)  at (6,  2) {$A(F)$};
	 							\node (BF)  at (8,  2) {$B(F)$};
	 							\node (CF)  at (10,  2) {\qquad\qquad\quad$C(F)=\QQ(\sqrt{p }, \sqrt {q},  \sqrt{s})$};%
	 							
	 							\node (F2)  at (-0.5,  4) {$F_2$};
	 							\node (A1F)  at (6,  4) {$A_1(F)$};
	 							\node (B1F)  at (8,  4) {$B_1(F)$};
	 							\node (C1F)  at (10,  4) {$\qquad\qquad\qquad\qquad C_1(F)=\QQ(\sqrt{p }, \sqrt {q},  \sqrt{s},\sqrt{2})$};
	 							\node (HF1)  at (15,  4) {$  F^{(1)} $};
	 							\node (HF2)  at (16.5,  6.5) {$  F^{(2)} $};
	 							
	 							\node (F3)  at (-0.5,  6) {$ F_3$};
	 							\node (A2F)  at (6,  6) {$A_2(F)$};
	 							\node (B2F)  at (8,  6) {$B_2(F)$};
	 							\node (C2F)  at (10,  6) {$\qquad\qquad\qquad\qquad\qquad\;  C_2(F)=\QQ(\sqrt{p }, \sqrt {q},  \sqrt{s},\sqrt{1+\sqrt{2}})$};
	 							\node (HF1SQRT2)  at (15,  6) {$F^{(1)}(\sqrt{2}) $};
	 							\node (HF2SQRT2)  at (16.5,  8.5) {$F^{(2)}(\sqrt{2}) $};

	 							\node (F3a)  at (-0.5,  7) {$\vdots$};
	 							\node (A2Fa)  at (6,  7) {$\vdots$};
	 							\node (B2Fa)  at (8,  7) {$\vdots$};
	 							\node (C2Fa)  at (10,  7) {$\vdots$};
	 							\node (HF1SQRT2a)  at (15,  7) {$\vdots$};
	 							\node (HF1SQRT2aa)  at (15,  8) {$\vdots$};
	 							\node (HF1SQRT2ab)  at (15,  9) {$\vdots$};
	 							\node (HF1SQRT2abc)  at (16.5,  9) {$\vdots$};
	 							\node (HF2inftySQRT2)  at (16.5,  10) {$F_\infty^{(2)} $};

	 							\node (F3ab)  at (-0.5,  9) {$\vdots$};
	 							\node (A2Fab)  at (6,  9) {$\vdots$};
	 							\node (B2Fab)  at (8,  9) {$\vdots$};
	 							\node (C2Fab)  at (10,  9) {$\vdots$};

	 							\node (F3ac)  at (-0.5,  8) {$\vdots$};
	 							\node (A2Fac)  at (6,  8) {$\vdots$};
	 							\node (B2Fac)  at (8,  8) {$\vdots$};
	 							\node (C2Fac)  at (10,  8) {$\vdots$};

	 							\node (Finfy)  at (-0.5, 10) {$ F_\infty$};
	 							\node (AinfyF)  at (6,  10) {$A_\infty(F)$};
	 							\node (BinfyF)  at (8,  10) {$B_\infty(F)$};
	 							\node (CinfyF)  at (10,  10) {$C_\infty(F)$};
	 							\node (HnfyF1SQRT2)  at (15,  10) {$F_\infty^{(1)}$};
	 							
	 							
	 							\draw[dashed,thick]  (F) --(F1) node[scale=0.4,  midway,  below right]{ {\;\; \bf  }};
	 							\draw (K) --(F1)  node[scale=0.4,  midway,  below right]{};
	 							
	 							\draw (F) --(AF)  node[scale=0.4,  midway,  below right]{};
	 							\draw (F) --(BF)  node[scale=0.4,  midway,  below right]{};
	 							\draw (F) --(CF)  node[scale=0.4,  midway,  below right]{};
	 							
	 							\draw (K) --(CK)  node[scale=0.4,  midway,  below right]{};
	 							\draw (CK) --(C1F)  node[scale=0.4,  midway,  below right]{};
	 							
	 							\draw (F1) --(A1F)  node[scale=0.4,  midway,  below right]{};
	 							\draw (F1) --(B1F)  node[scale=0.4,  midway,  below right]{};
	 							\draw (F1) --(C1F)  node[scale=0.4,  midway,  below right]{};
	 							
	 							\draw (CF) --(HF1)  node[scale=0.4,  midway,  below right]{};
	 							\draw (C1F) --(HF1SQRT2)  node[scale=0.4,  midway,  below right]{};

	 							\draw (AF) --(HF1)  node[scale=0.4,  midway,  below right]{};
	 							\draw (BF) --(HF1)  node[scale=0.4,  midway,  below right]{};
	 							
	 							\draw (A1F) --(HF1SQRT2)  node[scale=0.4,  midway,  below right]{};
	 							\draw (B1F) --(HF1SQRT2)  node[scale=0.4,  midway,  below right]{};

	 							\draw[dashed,thick]  (F1) --(F2) node[scale=0.4,  midway,  below right]{ {\;\; \bf  }};
	 							\draw[dashed,thick]  (AF) --(A1F) node[scale=0.4,  midway,  below right]{ {\;\; \bf  }};
	 							\draw[dashed,thick]  (BF) --(B1F) node[scale=0.4,  midway,  below right]{ {\;\; \bf  }};	  
	 							\draw[dashed,thick]  (CF) --(C1F) node[scale=0.4,  midway,  below right]{ {\;\; \bf  }};
	 							
	 							\draw[dashed,thick]  (F2) --(F3) node[scale=0.4,  midway,  below right]{ {\;\; \bf  }};
	 							\draw[dashed,thick]  (A1F) --(A2F) node[scale=0.4,  midway,  below right]{ {\;\; \bf  }};
	 							\draw[dashed,thick]  (B1F) --(B2F) node[scale=0.4,  midway,  below right]{ {\;\; \bf  }};	  
	 							\draw[dashed,thick]  (C1F) --(C2F) node[scale=0.4,  midway,  below right]{ {\;\; \bf  }};
	 							
	 							\draw[dashed,thick]  (HF1) --(HF1SQRT2) node[scale=0.4,  midway,  below right]{ {\;\; \bf  }};

	 							\draw (F2) --(A2F)  node[scale=0.4,  midway,  below right]{};
	 							\draw (F2) --(B2F)  node[scale=0.4,  midway,  below right]{};
	 							\draw (F2) --(C2F)  node[scale=0.4,  midway,  below right]{};
	 							
	 							\draw (HF1) --(HF2)  node[scale=0.4,  midway,  below right]{};
	 							\draw (HF1SQRT2) --(HF2SQRT2)  node[scale=0.4,  midway,  below right]{};
	 							\draw[dashed,thick]  (HF2) --(HF2SQRT2) node[scale=0.4,  midway,  below right]{ {\;\; \bf  }};
	 					\end{tikzpicture}}
	 					
	 				\end{figure}
	 			\end{center}
	 		\end{minipage}
	 		
	 }}
	 

	 Now we can proof our main theorem.

	 \begin{proof}[Proof of Theorem \ref{maintheorem} $($The Main Theorem$)$]
	 	According to Lemma \ref{cn12hpqs},   the $2$-class group of $ C(F)=\QQ(\sqrt{p }, \sqrt {q},  \sqrt{s}  )$ is cyclic of order $ \frac{1}{2}h_2(pqs)$. So if $h_2(pqs)=4$, then,   by Remark \ref{rem2} and Figure \ref{fig2}, the $2$-class groups of $A(F)$ and $B(F)$ are cyclic. If furthermore, the $2$-class number of $C(K)=\QQ(\sqrt {2p},  \sqrt{2s},\sqrt{q})$ equals to $4$, then the Hilbert $2$-class field tower of $F_1$  terminates at the first layer which is $F^{(1)}(\sqrt{2}) $. Therefore, the $2$-class number of $C_1(F)= \QQ(\sqrt{p }, \sqrt {q},  \sqrt{s},\sqrt{2})$ is equal to $2$. So    for all $n\geq 0$, the $2$-class number of  $C_n(F)= \QQ(\sqrt{p }, \sqrt {q},  \sqrt{s},2\cos( {2\pi}/{2^{n+2}}))$ is equal to $2$ (cf. Lemma \ref{lm fukuda}). 
	 	It follows from Figure \ref{fig1} that for all  $n\geq 0$, $ \mathrm{Gal}(\mathcal{L}(F_n)/F_n)$ is abelian.
	 	
	 	Now assume that $h_2(pqs)=8$  and that the $2$-class number of $C(K)$ equal $8$. Then 
	 	$2$-class numbers of $A(F) $ and $B(F)$ (resp.  $A_1(F) $ and $B_1(F)$) are equal  to $4$. It follows that for all $n\geq 1$, the $2$-class number  of  $A_n(F)= A_1(F)(2\cos( {2\pi}/{2^{n+2}}))$ is equal to $4$  (cf. Lemma \ref{lm fukuda}). So by Figure \ref{fig2}, for all $n\geq 0$, $ \mathrm{Gal}(\mathcal{L}(F_n)/F_n)$ is not abelian, more precisely, quaternion or dihedral  of order $8$.
	 \end{proof}
	 
	 \medskip
	 
	Finally,  Proposition \ref{mainprop} is a direct deduction from the first part of the proof and   Figures \ref{fig1}, \ref{fig2} and  \ref{fig3}. 
	 
	 \medskip

	 Using Pari/GP calculator software, we give the following prime numbers that satisfy the conditions of the main theorem.


	 \begin{table}[H]
	 	
	 	{  {\centering\small \centering\renewcommand{\arraystretch}{1.6}
	 			\setlength{\tabcolsep}{0.5cm}
	 			$$\begin{tabular}{c|c|c|c|c|c|c}
	 				\hline
	 				$p$& $q$ & $s$ &  $\left(\frac{	p}{	q}\right)$ &$\left(\frac{p}{s}\right)$ & $h_2(pqs)$ & $h_2(C(K))$ \\
	 				\hline
	 				13& 43 & 3 &1&1& 4 & 4 \\
	 				\hline
	 				61& 83 & 3 &1 &1&4 & 4 \\
	 				\hline
	 				29& 59 & 7 & 1&1&8 & 8 \\
	 				\hline
	 				53& 59 & 7 &1 &1&4 & 4 \\
	 				\hline
	 				37& 67 & 11 &1&1& 8 & 8 \\
	 				\hline
	 				29& 83 & 23 &1&1& 8 & 8 \\
	 				\hline
	 				29& 83 & 67 &1&1& 4 & 4 \\
	 				\hline
	 			\end{tabular} $$ }}
	 	
	 	\caption{Some primes satisfying the main theorem conditions}	 
	 \end{table}

	 \begin{remark}
	 	Since the dissemination of the preliminary preprint of this article, several recent papers (see \cite{ChemseddinGreenbergConjectureI,chemshamza,jerrair2,laxannoppre}) have cited it. Notably, the authors of \cite{jerrair2} obtained a partial improvement of the main theorem presented in this paper.
	 \end{remark}
	 
	 \section*{Acknowledgment}
I am grateful to the reviewers for their insightful comments, which have contributed to improving this paper.

 \end{document}